\documentclass[oneside,english]{amsart}
\usepackage[T1]{fontenc}
\usepackage[latin9]{inputenc}
\usepackage{color}
\usepackage{float}
\usepackage{amstext}
\usepackage{amsthm}
\usepackage{amssymb}

\makeatletter
\numberwithin{equation}{section}
\numberwithin{figure}{section}
\theoremstyle{plain}
\newtheorem{thm}{\protect\theoremname}
\theoremstyle{definition}
\newtheorem{defn}[thm]{\protect\definitionname}
\theoremstyle{remark}
\newtheorem*{rem*}{\protect\remarkname}
\theoremstyle{definition}
\newtheorem{example}[thm]{\protect\examplename}
\theoremstyle{plain}
\newtheorem{prop}[thm]{\protect\propositionname}
\theoremstyle{plain}
\newtheorem*{prop*}{\protect\propositionname}

\author{Lin Jiu}
\address{Department of Mathematics and Statistics, Dalhousie University, 
6316 Coburg Road, Halifax, Nova Scotia, Canada B3H 4R2}
\email{Lin.Jiu@dal.ca}

\author{Diane Y.H. Shi*}
\thanks{*Corresponding author}
\address{School of Mathematics, Tianjin University, No. 92 Weijin Road, 
Nankai District, Tianjin 300072, P. R. China}
\email{shiyahui@tju.edu.cn}

\date{}

\makeatother

\usepackage{babel}
\providecommand{\definitionname}{Definition}
\providecommand{\examplename}{Example}
\providecommand{\propositionname}{Proposition}
\providecommand{\remarkname}{Remark}
\providecommand{\theoremname}{Theorem}

\begin{document}
\title{On $b$-ary binomial coefficients with negative entries}
\begin{abstract}
We generalize the $b$-ary binomial coefficients with negative entries,
which is based on the generating function obtained in early work.
Besides an explicit expression involving the restricted partition,
several properties such as symmetry, congruence and Pascal-like recurrence
are studies. Finally, we also provide two different generalizations,
partially satisfying Pascal-like recurrences. 
\end{abstract}

\maketitle

\section{Introduction}

Callan \cite[Thm.~2]{Binomial} originally extended the classical
binomial identity 
\[
(X+Y)^{n}=\sum_{k=0}^{n}\binom{n}{k}X^{k}Y^{n-k},
\]
to 
\begin{equation}
(X+Y)^{S_{2}(n)}=\sum_{0\leq k\lesssim_{2}n}X^{S_{2}(k)}Y^{S_{2}(n-k)},\label{eq:S2n}
\end{equation}
involving the binary expansions of $n$ and $k$. Here, given any
positive integer $n$, we denote its $b$-ary expansion as 
\[
n=\sum_{l=0}^{N-1}n_{l}b^{l}=\left(n_{N-1}\cdots n_{0}\right)_{b}.
\]
Then, $S_{b}(n):=n_{N-1}+\cdots+n_{0}$ is the sum of all the digits
of $n$, in base $b$. In addition, $0\leq k\lesssim_{b}n$ means
the sum index $k$ runs over all integers from $0$ to $n$ such that
the $b$-ary addition $k+(n-k)=n$ is carry-free.

An extension of (\ref{eq:S2n}), to any base $b$, is obtained as
\cite[Eq.~10]{bAry}
\begin{equation}
(X+Y)^{S_{b}(n)}=\sum_{k=0}^{n}\binom{n}{k}_{b}X^{S_{b}(k)}Y^{S_{b}(n-k)},\label{eq:Sbn}
\end{equation}
where, the $b$-ary binomial coefficients is defined by \cite[Eq.~11]{bAry}
\begin{equation}
\binom{n}{k}_{b}=\overset{N-1}{\underset{l=0}{\prod}}{n_{l} \choose k_{l}},\label{eq:BinomialCoefficientsBaseb}
\end{equation}
for nonnegative integers $n$ and $k$. Here, we assume $k=(k_{N-1}\cdots k_{0})_{b}$
and $N=\min\left\{ m\in\mathbb{N}:n_{s}=k_{s}=0\text{ for all }s\geq m\right\} $.
Namely, if $n$ has $N_{1}$ digits and $k$ has $N_{2}$ digits,
in base $b$, then $N=\max\left\{ N_{1},N_{2}\right\} $. This setup
for $N$ shall be applied throughout this paper. It can be observed
that the carry-free condition, appeared in (\ref{eq:S2n}), is eliminated
in (\ref{eq:Sbn}), due to the definition (\ref{eq:BinomialCoefficientsBaseb}).
Moreover, the generating function of the $b$-ary binomial coefficients
is obtained as \cite[Eq.~13]{bAry}
\begin{equation}
\sum_{k=0}^{n}\binom{n}{k}_{b}x^{k}=\prod_{l=0}^{N-1}\left(1+x^{b^{l}}\right)^{n_{l}}.\label{eq:bAryGF}
\end{equation}

In an early paper, Loeb \cite[Thm.~4.1]{NegativeBinomial} defined
in general the binomial coefficients with negative integer entries:
for $n,k\in\mathbb{Z}$, 
\[
\binom{n}{k}:=\lim_{\varepsilon\rightarrow0}\frac{\Gamma\left(n+1+\varepsilon\right)}{\Gamma\left(k+1+\varepsilon\right)\Gamma\left(n-k+1+\varepsilon\right)},
\]
which also admit a combinatorial interpretation, counting the number
of elements in a hybrid set. Alternatively, it can be defined as the
coefficient of $x^{k}$ in the power series of $\left(1+x\right)^{n}$
\cite[Prop.~4.5]{NegativeBinomial}:
\begin{equation}
\binom{n}{k}:=\left[x^{k}\right]\left(1+x\right)^{n},\label{eq:NegativeBinomialDEF}
\end{equation}
where if $k$ is negative, the inverse power series is applied. More
precisely, for positive $n$, the following two series expansions
will be considered
\begin{equation}
(1+x)^{-n}=\sum_{j=0}^{\infty}\binom{-n}{j}x^{j}=\sum_{j=n}^{\infty}\binom{-n}{-j}x^{-j}.\label{eq:ExpansionBinomialNegative}
\end{equation}
The first three cases, $n=1$, $2$, and $3$, are listed here. Calculations
in later examples will consult these expressions: 
\begin{eqnarray*}
(1+x)^{-1}= & {\displaystyle \sum_{j=0}^{\infty}}(-1)^{j}x^{j} & ={\displaystyle \sum_{j=1}^{\infty}}(-1)^{j+1}x^{-j},\\
(1+x)^{-2}= & {\displaystyle \sum_{j=0}^{\infty}}(-1)^{j}(j+1)x^{j} & ={\displaystyle \sum_{j=2}^{\infty}}(-1)^{j}(j-1)x^{-j},\\
(1+x)^{-3}= & {\displaystyle \sum_{j=0}^{\infty}}(-1)^{j}\frac{(j+1)(j+2)}{2}x^{j} & ={\displaystyle \sum_{j=3}^{\infty}}\frac{(-1)^{j+1}(j-1)(j-2)}{2}x^{-j}.
\end{eqnarray*}

The main purpose of this work is to generalize the $b$-ary binomial
coefficients with negative entries, similarly as (\ref{eq:NegativeBinomialDEF}).
Definition, examples and an explicit expression are introduced in
Section \ref{sec:DEF}. In Section \ref{sec:Properties}, we shall
study some properties, such as symmetry, congruence, the Chu-Vandermonde
identity and the Pascal-like recurrence. In Section \ref{sec:Alternatives},
we briefly discuss two other natural but different generalizations,
partially satisfying Pascal-like recurrences. 

\section{\label{sec:DEF}Definition and explicit expression}

First of all, we adopt the convention that a negative integer has
all its digits nonpositive, in any base $b$. Namely, if $n=(n_{N-1}n_{N-2}\cdots n_{1}n_{0})_{b}>0$,
then 
\[
-n=\left((-n_{N-1})(-n_{N-2})\cdots(-n_{1})(-n_{0})\right)_{b},
\]
which is compatible with the $b$-ary expansion that 
\[
-n=(-n_{N-1})b^{N-1}+(-n_{N-2})b^{N-2}+\cdots+(-n_{1})b+(-n_{0}).
\]
It also indicates
\[
S_{b}(-n)=-n_{N-1}-n_{N-2}-\cdots-n_{0}=-S_{b}(n).
\]

We now extend the $b$-ary binomial coefficients with negative entries
as follows. 
\begin{defn}
Let $n,k\in\mathbb{Z}$ with $n=\left(n_{N-1}\cdots n_{0}\right)_{b}$
and $k=\left(k_{N-1}\cdots k_{0}\right)_{b}$.
\begin{equation}
\binom{n}{k}_{b}:=\left[x^{k}\right]\prod_{l=0}^{N-1}\left(1+x^{b^{l}}\right)^{n_{l}},\label{eq:DEFnkb}
\end{equation}
where, when both $n$ and $k$ are negative, it is the coefficient
of $x^{k}$ of the inverse power series of the right-hand side. 
\end{defn}

\begin{rem*}
For simplicity, we shall denote the generating function by
\[
f_{n,b}(x):=\prod_{l=0}^{N-1}\left(1+x^{b^{l}}\right)^{n_{l}}.
\]
Also, for the negative case, we assume $n$ is positive and consider
the expansions
\begin{equation}
\sum_{k=0}^{\infty}\binom{-n}{k}_{b}x^{k}=f_{-n,b}(x)=\sum_{k=1}^{\infty}\binom{-n}{-k}_{b}\frac{1}{x^{k}}.\label{eq:GFnegativen}
\end{equation}
\end{rem*}
\begin{example}
\label{exa:Example}Let $b=4$ and $n=6$, so that $-n=-6=\left((-1)(-2)\right)_{4}$
and
\[
f_{-6,4}(x)=\frac{1}{\left(1+x^{4}\right)\left(1+x\right)^{2}}.
\]
(1) For $k=7$, since as $x\rightarrow0$, 
\[
f_{-6,4}(x)=1-2x+3x^{2}-4x^{3}+4x^{4}-4x^{5}+4x^{6}-4x^{7}+O\left(x^{8}\right),
\]
we see 
\[
\binom{-6}{7}_{4}=-4.
\]
(2) For $k=-8$, as $x\rightarrow\infty$, 
\[
f_{-6,4}(x)=x^{-6}-2x^{-7}+3x^{-8}+O\left(x^{-9}\right)\Longrightarrow\binom{-6}{-8}_{4}=3.
\]
\end{example}

The next proposition gives an explicit expression of the $b$-ary
binomial coefficients with negative entries. 
\begin{prop}
\label{prop:ExplicitExpression} Let $n=\left(n_{N-1}\cdots n_{0}\right)_{b}$
be positive. Then, 

\begin{equation}
\binom{-n}{k}_{b}=\begin{cases}
{\displaystyle \sum_{\left(j_{N-1},\ldots,j_{0}\right)\in\mathcal{P}_{n}(k,\mathbf{b}_{N})}\prod_{l=0}^{N-1}\binom{-n_{l}}{j_{l}},} & \text{if }k\geq0;\\
{\displaystyle \sum_{\left(j_{N-1},\ldots,j_{0}\right)\in\mathcal{P}_{n}^{*}(-k,\mathbf{b}_{N})}\prod_{l=0}^{N-1}\binom{-n_{l}}{-j_{l}},} & \text{if }k<0,
\end{cases}\label{eq:nbnegativek2}
\end{equation}
where $\bullet$ $\mathbf{b}_{N}:=\{1,b,\ldots,b^{N-1}\}$;

\hspace{16.5bp}$\bullet$ $\mathcal{P}_{n}(k,\mathbf{b}_{N})$ is
the set of \emph{restricted partitions} of $k$ into parts in $\mathbf{b}_{N}$,
i.e., 

\hspace{24bp}$N$-tuples of nonnegative integers $\left(j_{N-1},\ldots,j_{0}\right)$
such that 
\begin{equation}
j_{N-1}b^{N-1}+\cdots+j_{1}b^{1}+j_{0}=k;\label{eq:Expansionkbary}
\end{equation}

\hspace{16.5bp}$\bullet$ and $\mathcal{P}_{n}^{*}(-k,\mathbf{b}_{N})$
is the subset of $\mathcal{P}_{n}(-k,\mathbf{b}_{N})$, containing
all the $N$-tuples 

\hspace{24bp}with extra restrictions: $j_{l}\geq n_{l}$, for $l=0,\ldots,N-1$. 
\end{prop}

\begin{rem*}
If $b>\max\left\{ n,\left|k\right|\right\} $, then both $n$ and
$k$ have one digit, i.e., $N=1$. In this case, (\ref{eq:nbnegativek2})
reduces to (\ref{eq:NegativeBinomialDEF}).
\end{rem*}
\begin{example}
Reconsider, in Example \ref{exa:Example}, that $b=4$ and $-n=-6=\left((-1)(-2)\right)_{4}$.
Then, $N=2\Rightarrow\mathbf{4}_{2}=\left\{ 1,4\right\} $. 

\smallskip{}

\noindent(1) For $k=7=1\cdot4+3\cdot1=0\cdot4+7\cdot1$, we see $\mathcal{P}_{6}(7,\mathbf{4}_{2})=\left\{ (1,3),(0,7)\right\} $.
By (\ref{eq:nbnegativek2}),
\[
\binom{-6}{7}_{4}=\binom{-1}{1}\binom{-2}{3}+\binom{-1}{0}\binom{-2}{7}=(-1)\cdot(-4)+1\cdot(-8)=-4.
\]
(2) For $k=-8$, it is not hard to see that $\mathcal{P}_{6}^{*}\left(8,\mathbf{4}_{2}\right)=\left\{ (1,4)\right\} $.
Therefore, 
\[
\binom{-6}{-8}_{4}=\binom{-1}{-1}\binom{-2}{-4}=1\cdot3=3.
\]
\end{example}

\begin{proof}
[Proof of Proposition \ref{prop:ExplicitExpression}](1) If $k>0$,
we apply (\ref{eq:ExpansionBinomialNegative}) to each factor of $f_{-n,b}(x)$:
\[
f_{-n,b}(x)=\prod_{l=0}^{N-1}\left(\sum_{j_{l}=0}^{\infty}\binom{-n_{l}}{j_{l}}x^{j_{l}b^{l}}\right)=\sum_{j_{0},\ldots,j_{N-1}=0}^{\infty}\left(\prod_{l=0}^{N-1}\binom{-n_{l}}{j_{l}}\right)x^{\overset{N-1}{\underset{l=0}{\sum}}j_{l}b^{l}}.
\]
By comparing coefficients, the first case in (\ref{eq:nbnegativek2})
is confirmed. \smallskip{}
\\
(2) Similarly if $k<0$, by the inverse power series in (\ref{eq:ExpansionBinomialNegative}),
\[
f_{-n,b}(x)=\prod_{l=0}^{N-1}\left(\sum_{j_{l}=n_{l}}^{\infty}\binom{-n_{l}}{-j_{l}}x^{-j_{l}}\right)=\sum_{j_{l}=n_{l}}^{\infty}\left(\prod_{l=0}^{N-1}\binom{-n_{l}}{-j_{l}}\right)x^{-\overset{N-1}{\underset{l=0}{\sum}}j_{l}b^{l}}.
\]
Note that all the $j_{l}'s$, $l=0,\ldots,N-1$ start from $n_{l}$
(, rather than $0$ in the case of $k\geq0$). We need to additionally
restrict the $N$-tuples $\left(j_{N-1},\ldots,j_{0}\right)\in\mathcal{P}_{n}^{*}(-k,\mathbf{b}_{N})$,
as desired.
\end{proof}
\begin{rem*}
When both $n$ and $k$ are positive, $f_{n,b}(x)$ is a polynomial
\[
f_{n,b}(x)={\displaystyle \prod_{l=0}^{N-1}}\left(1+x^{b^{l}}\right)^{n_{l}}=\prod_{l=0}^{N-1}\left(\sum_{j_{l}=0}^{n_{l}}\binom{n_{l}}{j_{l}}x^{j_{l}b^{l}}\right).
\]
Then, it requires $0\leq j_{l}\leq n_{l}$, for $l=0,\ldots,N-1$,
which restricts (\ref{eq:Expansionkbary}) to have only one solution:
the unique expression of $k$ in base $b$: $j_{l}=k_{l}$. This explains
the reason that (\ref{eq:BinomialCoefficientsBaseb}) and (\ref{eq:bAryGF})
define the same coefficients, when $n\geq0$. In the last section,
we shall see that a natural generalization of (\ref{eq:BinomialCoefficientsBaseb})
for $n<0$ is different from (\ref{eq:DEFnkb}). 
\end{rem*}

\section{\label{sec:Properties}Properties}

\subsection{Symmetry}
\begin{prop*}
For any $n,k\in\mathbb{Z}$, ${\displaystyle \binom{n}{k}_{b}=\binom{n}{n-k}_{b}}$.
\end{prop*}
\begin{proof}
Since the non-negative case is already proven in \cite[Thm.~10]{bAry},
we let $n,k\in\mathbb{N}$ and it suffices to show that 
\[
\binom{-n}{k}_{b}=\binom{-n}{-n-k}_{b}\ \ \ \text{and}\ \ \ \binom{-n}{-k}_{b}=\binom{-n}{-n+k}_{b}.
\]
(1) By the left expansion in (\ref{eq:GFnegativen}), 
\[
\binom{-n}{k}_{b}=\frac{f_{-n,b}^{(k)}(0)}{k!}.
\]
On the other hand, note that
\[
f_{-n,b}\left(\frac{1}{x}\right)=\prod_{l=0}^{N-1}\left(1+x^{-b^{l}}\right)^{-n_{l}}=\prod_{l=0}^{N-1}x^{n_{l}b^{l}}\left(x^{b^{l}}+1\right)^{-n_{l}}=x^{n}f_{-n,b}(x).
\]
From the higher-order product rule that
\[
f_{-n,b}^{(n+k)}\left(\frac{1}{x}\right)=\sum_{l=0}^{n+k}\binom{n+k}{l}\frac{\mathrm{d}^{l}(x^{n})}{\mathrm{d}x^{l}}f_{-n,b}^{(n+k-l)}(x),
\]
we see, when letting $x\rightarrow0$, the only non-zero term on the
right-hand side is that $l=n$.  Thus, 
\[
\binom{-n}{-n-k}_{b}=\frac{1}{(n+k)!}\lim_{x\rightarrow0}f_{-n,b}^{(n+k)}\left(\frac{1}{x}\right)=\frac{1}{(n+k)!}\binom{n+k}{n}n!f_{-n,b}^{(k)}(0)=\binom{-n}{k}_{b},
\]
which proves the first symmetric identity.\smallskip{}
\\
(2) For the second identity, if $k\geq n\Leftrightarrow-n+k\geq0$,
by interchanging the two sides and by noting $-n-(-n+k)=-k$, it is
equivalent to the first identity. 

Now, The remaining case is that $k<n$, so that both $-k$ and $-n+k$
are negative. In fact, we have
\[
\binom{-n}{-k}_{b}=\binom{-n}{-n+k}_{b}=0,
\]
which can be seen either from (\ref{eq:nbnegativek2}), where in this
case $\mathcal{P}_{n}(-k,\mathbf{b}_{N})=\emptyset$; or from the
following direct calculation:
\[
\sum_{k=1}^{\infty}\binom{-n}{-k}_{b}\frac{1}{x^{k}}=f_{-n,b}(x)=\frac{1}{x^{n}}\prod_{l=0}^{N-1}\frac{1}{\left(1+\frac{1}{x^{b^{l}}}\right)^{n_{l}}}=\frac{1}{x^{n}}\prod_{l=0}^{N-1}\left(\sum_{j_{l}=1}^{\infty}\frac{(-1)^{j_{l}}}{x^{j_{l}b^{l}}}\right)^{n_{l}}.\qedhere
\]
\end{proof}
\begin{rem*}
The proof above suggests a slight modification of (\ref{eq:GFnegativen}):
\begin{equation}
\sum_{k=0}^{\infty}\binom{-n}{k}_{b}x^{k}=f_{-n,b}(x)=\sum_{k={\color{red}n}}^{\infty}\binom{-n}{-k}_{b}x^{-k}.\label{eq:GF-nkb}
\end{equation}
\end{rem*}

\subsection{Recurrence}

As proven in \cite[Thm.~10]{bAry}, for nonnegative integers $n$
and $k$, 
\[
\binom{n}{k}_{b}+\binom{n}{k-1}_{b}=\binom{n+1}{k}_{b}
\]
holds when $\binom{n+1}{k}_{b}\neq0$. The next proposition shows
that for negative $n$, the recurrence holds similarly, with some
(in-)divisibility restriction. 
\begin{prop}
\label{prop:Pascal} Let $n=\left(n_{N-1}\cdots n_{0}\right)>0$ and
$k\in\mathbb{Z}$. If $b\nmid n$, then
\begin{equation}
\binom{-n}{k}_{b}+\binom{-n}{k-1}_{b}=\binom{-n+1}{k}_{b}.\label{eq:Pascal}
\end{equation}
\end{prop}

\begin{proof}
Since $b\nmid n\Leftrightarrow n_{0}\neq0$, we have $-n+1=\left((-n_{N-1})\cdots(-(n_{0}-1)\right)_{b}$.
Then, 
\[
f_{-n+1,b}(x)=\left(1+x\right)^{-n_{0}+1}\prod_{l=1}^{N-1}\left(1+x^{b^{l}}\right)^{-n_{l}}=(1+x)f_{-n,b}(x).
\]
This recurrence gives (\ref{eq:Pascal}) by expanding both sides and
comparing coefficients of the two expansions in (\ref{eq:GF-nkb}). 
\end{proof}
\begin{rem*}
By a similar calculation, we can see that, if for some $s\in\left\{ 0,\ldots,N-1\right\} $,
$n_{s}\neq0$, then 
\begin{equation}
\binom{-n}{k}_{b}+\binom{-n}{k-b^{s}}_{b}=\binom{-n+b^{s}}{k}_{b}.\label{eq:bAryPascal}
\end{equation}
\end{rem*}
Next, we consider the case $b\mid n$, in the next proposition.
\begin{prop}
Let $n=\left(n_{N-1}\cdots n_{s}0\cdots0\right)>0$, i.e., $n_{0}=\cdots=n_{s-1}=0$
and $n_{s}\neq0$. Then, for any $m\in\left\{ 0,\ldots,s-1\right\} $
and $k\geq b^{s}$ or $k\leq-n+b^{m}$, 
\begin{equation}
\binom{-n+b^{s}}{k}_{b}=\sum_{\genfrac{}{}{0pt}{}{b^{m}\leq j\leq b^{s}}{b^{m}\mid j}}\binom{b^{s}-b^{m}}{j}_{b}\binom{-n+b^{m}}{k-j}_{b}.\label{eq:PartialChuVandermonde}
\end{equation}
\end{prop}

\begin{proof}
Since $n-b^{m}=\left(n_{N-1}\cdots(n_{s}-1)(b-1)\cdots(b-1)0\cdots0\right)_{b}$,
we see
\[
f_{-n+b^{m},b}(x)=\frac{{\displaystyle \prod_{l=m}^{s-1}}\left(1+x^{b^{l}}\right)^{1-b}}{\left(1+x^{b^{s}}\right)^{n_{s}-1}{\displaystyle \prod_{l=s+1}^{N-1}}\left(1+x^{b^{l}}\right)^{n_{l}}}=\frac{1+x^{b^{s}}}{{\displaystyle \prod_{l=m}^{s-1}\left(1+x^{b^{l}}\right)^{b-1}}}f_{-n,b}(x),
\]
namely, 
\begin{align*}
\left(1+x^{b^{s}}\right)f_{-n,b}(x) & =\left(\prod_{l=m}^{s-1}\left(1+x^{b^{l}}\right)^{b-1}\right)f_{-n+b^{m},b}(x)\\
 & =\prod_{l=m}^{s-1}\left(\sum_{j_{l}=0}^{b-1}\binom{b-1}{j_{l}}x^{j_{l}b^{l}}\right)f_{-n+b^{m},b}(x)\\
 & =\left(\sum_{\genfrac{}{}{0pt}{}{b^{m}\leq j\leq b^{s}}{b^{m}\mid j}}\binom{b^{s}-b^{m}}{j}_{b}x^{j}\right)f_{-n+b^{m},b}(x),
\end{align*}
where in the last step, we see that $j=\left(j_{N-1}\cdots j_{0}\right)_{b}$
runs over all integers between $b^{m}$ and $b^{s}$ with $j_{l}=0$,
$l=0,\ldots,m-1$, which is equivalent to $b^{m}\mid j$. Comparing
coefficients and applying (\ref{eq:bAryPascal}) to the left-hand
side complete the proof. 
\end{proof}

\subsection{Chu-Vandermonde identity}

Suppose both $n=\left(n_{N-1}\cdots n_{0}\right)_{b}$ and $m=\left(m_{N-1}\cdots m_{0}\right)_{b}$
are positive, and $n+m$ in base $b$ is carry-free. Then, 
\begin{equation}
f_{-(n+m),b}(x)=f_{-n,b}(x)f_{-m,b}(x),\label{eq:ChuVandermondeGF}
\end{equation}
which, by the series expansions, is
\[
\sum_{k=0}^{\infty}\binom{-n-m}{k}_{b}x^{k}=\sum_{k=0}^{\infty}\sum_{j=0}^{k}\binom{-n}{k-j}_{b}\binom{-m}{j}_{b}x^{k}.
\]
This leads to the following Chu-Vandermonde identities. 
\begin{prop}
For positive integers $n$ and $m$, such that $n+m$ in base $b$
is carry-free. Then, for $k\geq m$, we have 
\[
\binom{-n-m}{k}_{b}=\sum_{j=0}^{k}\binom{-n}{k-j}_{b}\binom{-m}{j}_{b}.
\]
and for $k\geq n+m$, 
\[
\binom{-n-m}{-k}_{b}=\sum_{j=1}^{k-1}\binom{-n}{-k+j}_{b}\binom{-m}{-j}_{b}.
\]
\end{prop}

\begin{rem*}
The second identity is obtained by the inverse power series of (\ref{eq:ChuVandermondeGF}),
which requires both $-k+j$ and $-j$ are negative. 
\end{rem*}
Next, we consider the mixed of positive and negative cases. Let both
$n,m$ be positive such that $n>m$ and $m+(n-m)$ is carry-free in
base $b$. Then, 
\[
f_{n-m,b}(x)=f_{n,b}(x)f_{-m,b}(x)\ \ \ \text{and}\ \ \ f_{-n+m,b}(x)=f_{-n,b}(x)f_{m,b}(x).
\]
From the first identity, we have 
\[
\sum_{k=0}^{n-m}\binom{n-m}{k}_{b}x^{k}=\left(\sum_{s=0}^{n}\binom{n}{s}_{b}x^{s}\right)\left(\sum_{j=0}^{\infty}\binom{-m}{j}_{b}x^{j}\right),
\]
which leads to 
\[
\binom{n-m}{k}_{b}=\sum_{j=0}^{k}\binom{n}{k-j}_{b}\binom{-m}{j}_{b}.
\]
Meanwhile, if we alternatively use 
\[
\sum_{k=0}^{n-m}\binom{n-m}{k}_{b}x^{k}=\left(\sum_{s=0}^{n}\binom{n}{s}_{b}x^{s}\right)\left(\sum_{j=m}^{\infty}\binom{-m}{-j}_{b}x^{-j}\right),
\]
we see
\[
\binom{n-m}{k}_{b}=\sum_{s=k+1}^{n}\binom{n}{s}_{b}\binom{-m}{k-s}_{b}.
\]
Similar discussion applies to the second identity. Therefore, we obtain
the following identities, whose proofs are omitted.
\begin{prop}
Given positive integers $n$ and $m$, such that $n>m$ and $m+(n-m)$
is carry-free in base $b$, we have\medskip{}
\\
(1) for $0\leq k\leq n-m$, 
\[
\binom{n-m}{k}_{b}=\sum_{j=0}^{k}\binom{n}{k-j}_{b}\binom{-m}{j}_{b}=\sum_{s=k+1}^{n}\binom{n}{s}_{b}\binom{-m}{k-s}_{b};
\]
(2) for $k\geq0$, 
\[
\binom{-n+m}{k}_{b}=\sum_{j=0}^{k}\binom{-n}{k-j}_{b}\binom{m}{j}_{b};
\]
(3) and for $k\geq n-m$, 
\[
\binom{-n+m}{-k}_{b}=\sum_{j=0}^{k}\binom{-n}{-k-j}_{b}\binom{m}{j}_{b}.
\]
\end{prop}

\subsection{Congruence}

As the authors \cite[Thm.~12]{bAry} pointed out, Lucas' congruence
theorem, for prime $p$, 
\[
\binom{n}{k}\equiv\binom{n}{k}_{p}=\overset{N-1}{\underset{l=0}{\prod}}{n_{l} \choose k_{l}}\pmod p
\]
becomes obvious, by letting $b=p$ in the generating function (\ref{eq:bAryGF})
and by using an elementary congruence
\[
(1+x)^{n}\equiv\overset{N-1}{\underset{l=0}{\prod}}\left(1+x^{p^{l}}\right)^{n_{l}}\pmod p.
\]
The reciprocal of the congruence above indicates that Lucas' congruence
theorem also holds for negative entries. 
\begin{prop}
For $n,k\in\mathbb{Z}$ and a prime $p$, ${\displaystyle \binom{n}{k}\equiv\binom{n}{k}_{p}}$
$\left(\bmod p\right)$. 
\end{prop}

\section{\label{sec:Alternatives}other possible generalizations}

Naturally, we could generalized the $b$-ary binomial coefficients
as, for positive integer $n$
\begin{equation}
\binom{-n}{k}_{b}^{*}:=\prod_{l=0}^{N-1}\binom{-n_{l}}{k_{l}},\label{eq:nkbStar}
\end{equation}
However, this definition is different from (\ref{eq:DEFnkb}). For
instance, (see also Example \ref{exa:Example}), 
\[
\binom{-6}{7}_{4}^{*}=\binom{-1}{1}\binom{-2}{3}=4\ \ \ \text{and}\ \ \ \binom{-6}{-8}_{4}^{*}=\binom{-1}{-2}\binom{-2}{0}=-1.
\]

Meanwhile, we tried to extend (\ref{eq:Sbn}), by expanding both sides
as 
\[
\left(X+Y\right)^{S_{b}(n)}=\prod_{l=0}^{N-1}\left(X+Y\right)^{n_{l}}=\prod_{l=0}^{N-1}\left(\sum_{j_{l}=0}^{n_{l}}\binom{n_{l}}{j_{l}}X^{-n_{l}-j_{l}}Y^{j_{l}}\right),
\]
while in the negative case
\[
\left(X+Y\right)^{-S_{b}(n)}=\prod_{l=0}^{N-1}X^{-n_{l}}\left(1+\frac{Y}{X}\right)^{-n_{l}}=\prod_{l=0}^{N-1}\left(\sum_{j_{l}=0}^{\infty}\binom{-n_{l}}{j_{l}}X^{-n_{l}-j_{l}}Y^{j_{l}}\right).
\]
When $Y$ has power $j_{N-1}+\cdots+j_{0}=S_{b}(k)$, $X$ has the
power $-S_{b}(n)-S_{b}(k)$. The inverse power series leads to similar
results. It seem that we could consider 
\[
\binom{-n}{k}_{b}^{**}:=\left[\frac{Y^{S_{b}(k)}}{X^{S_{b}(n)+S_{b}(k)}}\right]\left(X+Y\right)^{-S_{b}(n)}=\sum_{\genfrac{}{}{0pt}{}{j_{l}\geq0\text{ for }l=0,\ldots,N-1}{j_{N-1}+\cdots+j_{0}=S_{b}(k)}}\left(\prod_{l=0}^{N-1}\binom{-n_{l}}{j_{l}}\right),
\]
and 
\[
\binom{-n}{-k}_{b}^{**}:=\left[\frac{X^{S_{b}(k)}}{X^{S_{b}(n)}Y^{S_{b}(k)}}\right]\left(X+Y\right)^{-S_{b}(n)}=\sum_{\genfrac{}{}{0pt}{}{j_{l}\geq n_{l}\text{ for }l=0,\ldots,N-1}{j_{N-1}+\cdots+j_{0}=S_{b}(k)}}\left(\prod_{l=0}^{N-1}\binom{-n_{l}}{-j_{l}}\right).
\]
Not only do they have different values, e.g., by considering the two
series of $(1+x)^{-3}$, 
\[
\binom{-6}{7}_{4}^{**}=15\ \ \ \text{and}\ \ \ \binom{-6}{-8}_{4}^{**}=0,
\]
but also they do not extend (\ref{eq:Sbn}), since when $S_{b}(k')=S_{b}(k)$,
the definition above implies $\binom{-n}{k}_{b}^{**}=\binom{-n}{k'}_{b}^{**}$.
In the positive case, from the two expansions
\[
\sum_{j=0}^{S_{b}(n)}\binom{S_{b}(n)}{j}X^{j}Y^{S_{b}(n)-j}=(X+Y)^{S_{b}(n)}=\sum_{k=0}^{n}\binom{n}{k}_{b}X^{S_{b}(k)}Y^{S_{b}(n-k)},
\]
we should have 
\[
\binom{S_{b}(n)}{j}=\sum_{\genfrac{}{}{0pt}{}{0\leq k\leq n}{S_{b}(k)=j}}\binom{n}{k}_{b}.
\]
Note that the finite sum on the left-hand side becomes an infinite
series in the negative case, which makes the extension of (\ref{eq:Sbn})
tricky. 

Although the two different extends do not satisfy basic properties
such as symmetry and congruence, they satisfy similar Pascal-like
recurrence, for positive $k$. 
\begin{prop}
\label{prop:PascalStars} For positive integers $n$ and $k$, if
$b\nmid k$ and $b\nmid n$, then
\[
\binom{-n}{k}_{b}^{*}+\binom{-n}{k-1}_{b}^{*}=\binom{-n+1}{k}_{b}^{*}\ \ \ \text{and}\ \ \ \binom{-n}{k}_{b}^{**}+\binom{-n}{k-1}_{b}^{**}=\binom{-n+1}{k}_{b}^{**}.
\]
\end{prop}

\begin{proof}
We shall make use of the result \cite[Prop.~4.4]{NegativeBinomial}
that the generalized binomial coefficients also satisfy the Pascal-like
recurrence:
\[
\binom{-n}{k}+\binom{-n}{k-1}=\binom{-n+1}{k}.
\]
Note that $b\nmid k$ and $b\nmid n$ are equivalent to $k_{0}\neq0\neq n_{0}$,
so that 
\[
k-1=(k_{N-1}\cdots k_{1}(k_{0}-1))_{b}\ \ \ \text{and}\ \ \ -n+1=\left((-n_{N-1})\cdots(-n_{1})(-n_{0}+1)\right)_{b}.
\]
(1) Directly, we have 
\begin{align*}
\binom{-n}{k}_{b}^{*}+\binom{-n}{k-1}_{b}^{*} & =\prod_{l=0}^{N-1}\binom{-n_{l}}{k_{l}}+\binom{-n_{0}}{k_{0}-1}\prod_{l=1}^{N-1}\binom{-n_{l}}{k_{l}}\\
 & =\left(\binom{-n_{0}}{k_{0}}+\binom{-n_{0}}{k_{0}-1}\right)\prod_{l=1}^{N-1}\binom{-n_{l}}{k_{l}}\\
 & =\binom{-n_{0}+1}{k_{0}}\prod_{l=1}^{N-1}\binom{-n_{l}}{k_{l}}=\binom{-n+1}{k}_{b}^{*}.
\end{align*}
(2) By definition, 
\[
\binom{-n}{k-1}_{b}^{**}=\sum_{\genfrac{}{}{0pt}{}{j_{l}'\geq0\text{ for }l=0,\ldots,N-1}{j_{N-1}'+\cdots+j_{0}'=S_{b}(k-1)}}\prod_{l=0}^{N-1}\binom{-n_{l}}{j_{l}},
\]
and 
\begin{align*}
\binom{-n}{k}_{b}^{**} & =\sum_{\genfrac{}{}{0pt}{}{j_{l}\geq0\text{ for }l=0,\ldots,N-1}{j_{N-1}+\cdots+j_{0}=S_{b}(k)}}\prod_{l=0}^{N-1}\binom{-n_{l}}{j_{l}}\allowdisplaybreaks\\
 & =\sum_{\genfrac{}{}{0pt}{}{j_{0}>0,j_{l}\geq0\text{ for }l=1,\ldots,N-1}{j_{N-1}+\cdots+j_{0}=S_{b}(k)}}\prod_{l=0}^{N-1}\binom{-n_{l}}{j_{l}}+\sum_{\genfrac{}{}{0pt}{}{j_{0}=0,j_{l}\geq0\text{ for }l=0,\ldots,N-1}{j_{N-1}+\cdots+j_{0}=S_{b}(k)}}\prod_{l=0}^{N-1}\binom{-n_{l}}{j_{l}}.
\end{align*}
Note that $S_{b}(k-1)=S_{b}(k)-1$ and if $j_{0}>0$, there is a one-to-one
correspondence 
\begin{eqnarray*}
\left\{ \left(j_{N-1},\ldots,j_{0}\right):\sum_{l=0}^{N-1}j_{l}=S_{b}(k)\right\}  & \longleftrightarrow & \left\{ \left(j_{N-1}',\ldots,j_{0}'\right):\sum_{l=0}^{N-1}j_{l}'=S_{b}(k-1)\right\} \\
\left(j_{N-1},\ldots,j_{0}\right) & \rightleftharpoons & \left(j_{N-1},\ldots,j_{0}-1\right).
\end{eqnarray*}
Thus, we can rewrite that 
\[
\binom{-n}{k-1}_{b}^{**}=\sum_{\genfrac{}{}{0pt}{}{j_{0}>0,j_{l}\geq0\text{ for }l=1,\ldots,N-1}{j_{N-1}+\cdots+j_{0}=S_{b}(k)}}\binom{-n_{0}}{j_{0}-1}\prod_{l=1}^{N-1}\binom{-n_{l}}{j_{l}}.
\]
In addition, when $j_{0}=0$, 
\[
\binom{-n_{0}}{j_{0}}=1=\binom{-n_{0}+1}{j_{0}}.
\]
Therefore, 
\begin{align*}
\binom{-n}{k}_{b}^{**}+\binom{-n}{k-1}_{b}^{**}= & \sum_{\genfrac{}{}{0pt}{}{j_{0}=0,j_{l}\geq0\text{ for }l=1,\ldots,N-1}{j_{N-1}+\cdots+j_{0}=S_{b}(k)}}\binom{-n_{0}+1}{j_{0}}\prod_{l=0}^{N-1}\binom{-n_{l}}{j_{l}}\\
 & +\sum_{\genfrac{}{}{0pt}{}{j_{0}>0,j_{l}\geq0\text{ for }l=1,\ldots,N-1}{j_{N-1}+\cdots+j_{0}=S_{b}(k)}}\left(\binom{-n_{0}}{j_{0}}+\binom{-n_{0}}{j_{0}-1}\right)\prod_{l=1}^{N-1}\binom{-n_{l}}{j_{l}}\\
= & \sum_{\genfrac{}{}{0pt}{}{j_{l}\geq0\text{ for }l=0,\ldots,N-1}{j_{N-1}+\cdots+j_{0}=S_{b}(k)}}\binom{-n_{0}+1}{j_{0}}\prod_{l=1}^{N-1}\binom{-n_{l}}{j_{l}}=\binom{-n+1}{k}_{b}^{**}.
\end{align*}
\end{proof}
\begin{rem*}
(1) Note that by definition, if $k-1$ has more digits than $n$,
the recurrence for $\binom{-n}{k}_{b}^{*}$ holds vacuously, since
all the three terms vanish. \smallskip{}
\\
(2)  When $k<0$, for ``most'' values of $k$, similar recurrences
hold for $\binom{-n}{k}_{b}^{*}$. It can be observed that if $b\nmid n$
and $k\equiv1$ $\left(\bmod p\right)$, Prop.~\ref{prop:PascalStars}
tend to fail. However, there are not all the exceptions. Please see
the following Table \ref{tab:2}, which lists the result of 
\[
\binom{-n}{-k}_{4}^{*}+\binom{-n}{-k-1}_{4}^{*}-\binom{-n+1}{-k-1}_{4}^{*},
\]
for $1\leq n\leq10$ and $1\leq k\leq20$. Non-zero values locate
where the Pascal-like recurrence fails. 

We shall leave it as part of our future work. \vspace{-10bp}

\begin{table}[H]
\[
\left(\begin{array}{rrrrrrrrrrrrrrrrrrr}
0 & 0 & 1 & 0 & 0 & 0 & 0 & 0 & 0 & 0 & 0 & 0 & 0 & 0 & 0 & 0 & 0 & 0 & 0\\
0 & 0 & -3 & 0 & 0 & 0 & 0 & 0 & 0 & 0 & 0 & 0 & 0 & 0 & 0 & 0 & 0 & 0 & 0\\
0 & 0 & 3 & 0 & 0 & 0 & 0 & 0 & 0 & 0 & 0 & 0 & 0 & 0 & 0 & 0 & 0 & 0 & 0\\
0 & 0 & 0 & 1 & 0 & 0 & -1 & -1 & 0 & 0 & 1 & 1 & 0 & 0 & 0 & 0 & 0 & 0 & 0\\
0 & 0 & 2 & 1 & 0 & 0 & 0 & -1 & 0 & 0 & 0 & 1 & 0 & 0 & 1 & 0 & 0 & 0 & 0\\
0 & 0 & -2 & 0 & 0 & 0 & -4 & 0 & 0 & 0 & 4 & 0 & 0 & 0 & -3 & 0 & 0 & 0 & 0\\
0 & 0 & 4 & 0 & 0 & 0 & 2 & 0 & 0 & 0 & -2 & 0 & 0 & 0 & 3 & 0 & 0 & 0 & 0\\
0 & 0 & -1 & -1 & 0 & 0 & 0 & 2 & 0 & 0 & -1 & -3 & 0 & 0 & -1 & 0 & 0 & 0 & 0\\
0 & 0 & 1 & 0 & 0 & 0 & 1 & 1 & 0 & 0 & -1 & -2 & 0 & 0 & -2 & 0 & 0 & 0 & 0\\
0 & 0 & -3 & 0 & 0 & 0 & 1 & 0 & 0 & 0 & -5 & 0 & 0 & 0 & 6 & 0 & 0 & 0 & 0
\end{array}\right)
\]
\vspace{-5bp}

\caption{\label{tab:2}${\displaystyle \binom{-n}{-k}_{4}^{*}+\binom{-n}{-k-1}_{4}^{*}-\binom{-n+1}{-k-1}_{4}^{*}}$
for $\protect\begin{aligned}1\protect\leq n\protect\leq10\protect\\
1\protect\leq k\protect\leq20
\protect\end{aligned}
$}
\end{table}
\end{rem*}

\section*{Acknowledgment}

The corresponding author is supported by the National Science Foundation
of China (No.~1140149). And the first author is supported by the
Natural Sciences and Engineering Research Council of Canada (No.~145628481).


\begin{thebibliography}{1}
\bibitem{Binomial}D.~Callan, Sierpinski\textquoteright s triangle
and the Prouhet-Thue-Morse word, preprint, 2006, http://arxiv.org/abs/math/0610932.

\bibitem{bAry}L.~Jiu and C.~Vignat, On binomial identities in arbitrary
bases, \emph{J.}~\emph{Integer Seq.}~\textbf{19} (2016), Article
16.5.5.

\bibitem{NegativeBinomial}D.~E.~Loeb, Sets with a negative number
of elements, \emph{Adv. Math.} \textbf{91} (1992), 64--74.
\end{thebibliography}
\end{document}